\RequirePackage{fix-cm}

\documentclass{amsart}

\usepackage{graphicx}

\usepackage{amsmath}
\usepackage{amsfonts}
\usepackage{amssymb}
\usepackage{wasysym}
\usepackage{enumerate}
\usepackage{longtable}
\usepackage{pgf,tikz}
\usepackage{cite, url}
\usepackage{extarrows}

\usetikzlibrary{arrows}

\theoremstyle{plain}
\newtheorem{theor}{Theorem}

\newtheorem{lemm}[theor]{Lemma}
\newtheorem{kor}[theor]{Corollary}
\theoremstyle{definition}
\newtheorem{definition}[theor]{Definition}

\newcommand \id {\ensuremath\mathrm{id}} 

\newcommand\blfootnote[1]{%
  \begingroup
  \renewcommand\thefootnote{}\footnote{#1}%
  \addtocounter{footnote}{-1}%
  \endgroup
}

\begin{document}

\title{Definability in the embeddability ordering of finite directed graphs, II }

\author{\'{A}d\'{a}m Kunos}

\maketitle

\begin{abstract}
We deal with first-order definability in the embeddability ordering $( \mathcal{D}; \leq)$ of finite directed graphs. 
A directed graph $G\in \mathcal{D}$ is said to be embeddable into $G' \in \mathcal{D}$ if there exists an injective graph homomorphism $\varphi \colon G \to G'$.
We describe the first-order definable relations of $( \mathcal{D}; \leq)$ using the first-order language of an enriched small category of digraphs. 
The description yields the main result of the author's paper \cite{Kunos2015} as a corrolary and a lot more.
For example, the set of weakly connected digraphs turns out to be first-order definable in $(\mathcal{D}; \leq)$. 
Moreover, if we allow the usage of a constant, a particular digraph $A$, in our first-order formulas,  then the full second-order language of digraphs becomes available.
\end{abstract}

\section{Introduction}

\blfootnote{In the beginning, this research was supported by T\'{A}MOP 4.2.4. A/2-11-1-2012-0001 ``National Excellence Program---Elaborating and operating an inland student and researcher personal support system''. This project was subsidized by the European Union and co-financed by the European Social Fund. Later, the author was supported by OTKA grant K115518.}

In 2009--2010 J. Je\v{z}ek and R. McKenzie published a series of papers 
\cite{Jezek2009_1, Jezek2010, Jezek2009_3, Jezek2009_4} 
in which they have examined (among other things) the first-order definability in the substructure orderings of finite mathematical structures with a given type and determined the automorphism group of these orderings. 
They considered finite semilattices \cite{Jezek2009_1}, ordered sets \cite{Jezek2010}, distributive lattices \cite{Jezek2009_3} and lattices \cite{Jezek2009_4}. 
Similar investigations \cite{Kunos2015, Wires2016, Ramanujam2016, Thinniyam2017} have emerged since. 
The current paper is one of such, a continuation of the author's paper \cite{Kunos2015} that dealt with the embeddability ordering of finite directed graphs. 
That whole paper centers around one main theorem. 
In the current paper we extend this theorem significantly.

Let us consider a nonempty set $V$ and a binary relation $E\subseteq V^2$. We call the pair $G=(V,E)$ a {\it directed graph} or just {\it digraph}.
The elements of $V(=V(G))$\label{qwqepwrpdg} and $E(=E(G))$\label{qweoweru3} are called the{ \it vertices} and {\it edges} of $G$, respectively.
The directed graph $G^T:=(V,E^{-1})$\label{aqweiqportw8} is called the {\it transpose} of $G$, where $E^{-1}$ denotes the inverse relation of $E$.
A digraph $G$ is said to be embeddable into $G'$, and we write $G\leq G'$, if there exists an injective homomorphism $\varphi : G \to G'$. 
Let $\mathcal{D}$\label{xcvshgw8} denote the set of isomorphism types of finite digraphs. 
It is easy to see that $\leq$ is a partial order on $\mathcal{D}$.

Let $(\mathcal{A},\leq)$ be an arbitrary poset.
An $n$-ary relation $R$ is said to be (first-order) definable in $(\mathcal{A},\leq)$ if there exists a first-order formula
$\Psi(x_1,x_2,\dots, x_n)$ with free variables $x_1,x_2,\dots, x_n$ in the language of partially ordered sets such that for any  $a_1,a_2,\dots, a_n\in \mathcal{A}$,
$\Psi(a_1,a_2,\dots, a_n)$  holds in $(\mathcal{A},\leq)$ if and only if $(a_1,a_2,\dots, a_n)\in R$.
A subset of $\mathcal{A}$ is definable if it is definable as a unary relation. An element $a\in \mathcal{A}$ is said to be definable if
the set $\{a\}$ is definable.
In the poset $(\mathcal{D},\leq)$ let $G\prec G'$ denote that $G'$ covers $G$. Obviously $\prec$ is a definable relation in $(\mathcal{D},\leq)$.
In \cite{Kunos2015}, the main result is
\begin{theor}[Theorem 2.38 \cite{Kunos2015}]\label{vxmknfsj}
 In the poset $(\mathcal{D};\leq)$, the set $\{G,G^T\}$ is first-order definable for all finite digraph $G\in \mathcal{D}$. 
 \end{theor}
This theorem is the best possible in the following sense. Observe, that $G\mapsto G^T$ is an automorphism of $(\mathcal{D};\leq)$. 
This implies that the digraphs $G$ and $G^T$ cannot be distinguished with first-order formulas of $(\mathcal{D};\leq)$.
What does Theorem \ref{vxmknfsj} tell about first-order definability in $(\mathcal{D};\leq)$? 
It tells the following 
\begin{kor}
A finite set $H$ of digraphs  is definable if and only if 
$$\forall G\in\mathcal{D}:  \;\; G\in H \Rightarrow G^T \in H.$$ 
\end{kor}
So the first-order definability of finite subsets in $(\mathcal{D};\leq)$ is settled. What about infinite subsets? 
One might ask if the set of weakly connected digraphs is first-order definable in $(\mathcal{D};\leq)$ as a standard model-theoretic argument shows that it is not definable in the first-order language of digraphs. 
The answer to this question appears to be out of reach with the result of \cite{Kunos2015}. 
In this paper we build the apparatus to handle some of such questions. 
In doing so we follow a path laid by Je\v{z}ek and McKenzie in \cite{Jezek2010}.
In particular, the set of weakly connected digraphs turns out to be definable.

Our method is the following. We add a constant---a particular digraph that is not isomorphic to its transpose---$A$ to the structure $(\mathcal{D};\leq)$ to get $(\mathcal{D};\leq, A)$. 
We define an enriched small category $\mathcal{CD}'$ and show that its first-order language is quite strong: it contains the full second-order language of digraphs. 
Finally, we show that first-order definability in $\mathcal{CD}'$ (after factoring by isomorphism) is equivalent to first-order definability in $(\mathcal{D};\leq, A)$. 
This result gives Theorem \ref{vxmknfsj} as an easy corollary and a lot more.

The paper offers two approaches for the proof of the main theorem. We either use the result of \cite{Kunos2015}, Theorem \ref{vxmknfsj}, and do not get it as a corollary but have a more elegant proof for our main result. Or we do not use it, instead we get it as a corollary but we have a little more tiresome proof for the main result. 

Section \ref{jelolesj} consists of a table of notations to help the reader to find the definitions of the many notations used in the paper which might get frustrating otherwise.

\section{Precise formulation of the main theorem and some display of its power}\label{pakdhtgft}

Once more, we emphasize that the approach we present in this section is from Je\v{z}ek and McKenzie \cite{Jezek2010}. \\
Let $[n]$ denote the set $\{1,2,\dots,n\}$ for all $n\in\mathbb{N}$.
Let us define the small category $\mathcal{CD}$\label{qoiue2833} of finite digraphs the following way. 
The set $\text{ob}(\mathcal{CD})$\label{dfjhawkd} of objects consists of digraphs on $[n]$ for some $n\in\mathbb{N}$.
For all $A,B\in \text{ob}(\mathcal{CD})$ let $\mathrm{hom}(A,B)$\label{1382721} consist of triples $f=(A,\alpha,B)$\label{yxcmmnbcvn16235} where $\alpha: A\to B$ is a homomorphism, meaning  $(x,y)\in E(A)$ implies $(\alpha(x),\alpha(y))\in E(B)$. 
Composition of morphisms are made the following way.  
For arbitrary objects $A,B,C\in \text{ob}(\mathcal{CD})$ if $f=(A,\alpha,B)$ and $g=(B,\beta,C)$, then 
$$fg=(A,\beta\circ\alpha, C).$$
It is easy to see that $f\in \mathrm{hom}(A,B)$ is injective if and only if for all $X\in \text{ob}(\mathcal{CD})$ 
$$ \forall g,h\in \mathrm{hom}(X,A): \;\; gf=hf\Leftrightarrow g=h.$$
Similarly $f\in \mathrm{hom}(A,B)$ is surjective is and only if for all $X\in \text{ob}(\mathcal{CD})$
$$ \forall g,h\in \mathrm{hom}(B,X): \;\; fg=fh\Leftrightarrow g=h.$$
These are first-order definitions in the (first-order) language of categories, hence in $\mathcal{CD}$, isomorphism and embeddability are first-order definable. This implies that all first-order definable relations in $(\mathcal{D},\leq)$ are definable in $\mathcal{CD}$ too. 
To put it more precisely, if $\rho\subseteq \mathcal{D}^n$ is an $n$-ary relation definable in $(\mathcal{D};\leq)$ then
$$\{(A_1,\dots,A_n): A_i\in \text{ob}(\mathcal{CD}), (\bar{A_1},\dots,\bar{A_n})\in \rho\}$$
is definable in $\mathcal{CD}$, where $\bar{A_i}$ denotes the isomorphism type of $A_i$. 

\begin{definition}\label{xmcvnbsjdiqawe} Let us introduce some objects and morphisms: 
\begin{displaymath}
\begin{split}
{\bf E}_1\in \text{ob}(\mathcal{CD})&: \;\; V({\bf E}_1)=[1], \; E({\bf E}_1)=\emptyset, \\
{\bf I}_2\in \text{ob}(\mathcal{CD})&: \;\; V({\bf I}_2)=[2], \; E({\bf E}_1)=\{(1,2)\}, \\
{\bf f}_1\in \mathrm{hom}({\bf E}_1, {\bf I}_2)&: \;\; {\bf f}_1=({\bf E}_1,\{(1, 1)\} ,{\bf I}_2), \\
{\bf f}_2\in \mathrm{hom}({\bf E}_1, {\bf I}_2)&: \;\; {\bf f}_2=({\bf E}_1,\{(1, 2)\} ,{\bf I}_2). 
\end{split}
\end{displaymath}
Adding these four constants to $\mathcal{CD}$ we get $\mathcal{CD}'$. 
\end{definition}

In the first-order language of $(\mathcal{D},\leq)$, formulas can only operate with the facts whether digraphs as a whole are embeddable into each other or not, the inner structure of digraphs is (officially) unavailable. 
In the first-order language of $\mathcal{CD'}$ though, we can capture embeddability (as we have seen above) but it is possible to capture the first-order language of digraphs too. 
The latter is far from trivial, but the following argument explains it.
For any $X\in \text{ob}(\mathcal{CD})$ the set of morphisms $\mathrm{hom}({\bf E}_1, X)$ is naturally bijective with the elements of $X$. 
Observe that if $f,g\in \mathrm{hom}({\bf E}_1, X)$ are  
$$f=({\bf E}_1,\{(1, x)\}, X), \;\; g=({\bf E}_1, \{(1, y)\}, X)\;\; (x,y\in V(X)),$$
then $(x,y)\in E(X)$ holds if and only if 
\begin{equation}\label{aoosudzhac}
\exists h\in \mathrm{hom}({\bf I}_2, X): \;\; {\bf f}_1h=f, \; {\bf f}_2h=g.
\end{equation}
To put it briefly, $X\cong CD_X$, where
$$V(CD_X)=\mathrm{hom}({\bf E}_1, X), \;\; E(CD_X)=\{(f,g): f,g\in \mathrm{hom}({\bf E}_1, X), \; (\ref{aoosudzhac}) \text{ holds}\}.$$
This shows how we can reach the inner structure of digraphs with the first-order language of $\mathcal{CD'}$. So the first-order language of $\mathcal{CD'}$ is much richer than that of $(\mathcal{D},\leq)$. 
We can go even further. 
One can show that the first-order language of  $\mathcal{CD'}$ can express the full second-order language of digraphs. To formulate this more precisely, the first-order language of $\mathcal{CD'}$ can express a language containing not only variables ranging over objects and morphisms of $\mathcal{CD'}$ but also
\begin{enumerate}[(I)]
\item quantifiable variables ranging over
\begin{enumerate}[(a)]
\item elements of any object,
\item arbitrary subsets of objects,
\item arbitrary functions between two objects,
\item\label{uajjshz} arbitrary subsets of products of finitely many objects (heterogenous relations),
\end{enumerate}
\item dependent variables giving the universe and the edge relation of an object,
\item the apparatus to denote
\begin{enumerate}[(a)]
\item edge relation between elements,
\item application of a function to an element,
\item\label{jhatztsvx} membership of a tuple of elements in a relation.
\end{enumerate}
\end{enumerate}

For example, let us see how (Ib), (\ref{uajjshz}) and (\ref{jhatztsvx}) can be ``modelled'' in $\mathcal{CD'}$. \\
Let us start with (Ib). Let $E_n\in \text{ob}(\mathcal{CD}')$ denote the empty digraph on $[n]$. 
The set
$$E=\{E_n\in \text{ob}(\mathcal{CD}'): n\in \mathbb{N}\}$$
is easily definable in $\mathcal{CD}'$. 
Let $A\in \text{ob}(\mathcal{CD}')$ be an arbitrary object and $S\subseteq A$ a subset of it.
Let $\gamma$ be a bijection $V(E_{|S|})\to S$. Let us define the morphism
$$p: E_{|S|} \to A, \;\; p(x)=\gamma(x) \;\; (x\in V(E_{|S|})).$$
It is easy to see that we represented the subset $S$ with the pair $(E_{|S|}, p)$. For example, an universal quantification over the subsets of $A$ would look like 
$$(\forall E_{|S|}\in E)(\forall p\in \mathrm{hom}(E_{|S|}, A)).$$    
Next, let us consider (\ref{uajjshz}). Let $A_1,\dots, A_n\in \text{ob}(\mathcal{CD}')$ be arbitrary objects and let 
$R\subseteq A_1\times\dots \times A_n$ be nonempty. 
Let $\pi_i(r)$ be the $i$th projection of $r\in R$. 
The functions $\pi_1,\dots,\pi_n$ ``determine'' the relation $R$ in the following sense: 
$$(a_1,\dots,a_n)\in R \;\; \Leftrightarrow \;\; \exists r\in R: \pi_i(r)=a_i \;\;(i=1,\dots,n). $$
We will represent the functions $\pi_i$ the following way. 
Let $\gamma: V(E_{|R|})\to R$ be a bijection. Let us define the morphisms $p_i$: 
$$p_i: E_{|R|}\to A_i, \;\; p_i(x)=\pi_i(\gamma(x)) \;\; (x\in V(E_{|R|}))$$    
It is easy to see that we represented the relation $R$ uniquely with $(E_{|R|},p_1,\dots,p_n)$. So an example of an existential quantification of type (\ref{uajjshz}) is
$$(\exists E_{|R|}\in E)(\exists p_1\in \mathrm{hom}(E_{|R|}, A_1))\dots (\exists p_n\in \mathrm{hom}(E_{|R|}, A_n)).$$
For (\ref{jhatztsvx}), an element of $A_1\times\dots \times A_n$ is represented with an element of
\begin{equation}\label{xckfgsd}
\mathrm{hom}(E_1, A_1) \times \dots \times  \mathrm{hom}(E_1, A_n)
\end{equation}
and if $(E_{|R|},p_1,\dots,p_n)$ belongs to $R\subseteq A_1\times\dots \times A_n$ and 
$(f_1, \dots, f_n)$, an element of \eqref{xckfgsd}, belongs to $x \in A_1\times\dots \times A_n$, then $x\in R$ can be expressed in the way
$$ (\exists f \in \mathrm{hom}(E_1, E_{|R|}))(fp_1=f_1 \mathrel{\wedge} \dots \mathrel{\wedge} fp_1=f_1).$$

Let ${\bf A}\in \text{ob}(\mathcal{CD})$\label{xcvnsbndasiiwq} denote the digraph $V({\bf A})=[3]$, $E({\bf A})=\{(1,3),(2,3)\}$. 
Now from the fact that in $\mathcal{CD}'$ isomorphism and embeddabbility are definable and from Theorem \ref{vxmknfsj}, the set 
$$\{X\in \text{ob}(\mathcal{CD}): X\cong {\bf A}\text{ or }  X\cong {\bf A}^T\}$$
is definable in $\mathcal{CD}'$. From this set, the formula
$$(\exists x\in X)(\forall y\in X)(y\neq x \; \Rightarrow \; (y,x)\in E(X))$$
chooses the set 
$$\{X\in \text{ob}(\mathcal{CD}): X\cong {\bf A}\}.$$
This shows that the first order language of $\mathcal{CD'}$ is stronger then the first-order language of $(\mathcal{D, \leq)}$ because in the latter, the isomorphism type of ${\bf A}$ is not definable as it is not isomorphic to its transpose. 

\begin{definition}\label{xcmvnyxjdaskdqiwo} By adding the isomorphism type of ${\bf A}$ as a constant to $(\mathcal{D},\leq)$ we get $(\mathcal{D};\leq, A)$. Let us denote this structure by $\mathcal{D}'$.
\end{definition}

We say that the relation $\rho\subseteq (\text{ob}(\mathcal{CD}))^n$ is {\it isomorphism invariant} if when for $A_i, B_i\in \text{ob}(\mathcal{CD})$, $A_i\cong B_i$ ($1\leq i\leq n$), then
$$(A_1,\dots,A_n)\in \rho \; \Leftrightarrow \; (B_1,\dots,B_n)\in \rho.$$
The set of isomorphism invariant relations of $\text{ob}(\mathcal{CD})$ is naturally bijective with the relations of $\mathcal{D}$. 
The main result of the paper is the following

\begin{theor}\label{9ahgfa} A relation is first-order definable in $\mathcal{D}'$ if and only if the corresponding isomorphism invariant relation of $\mathcal{CD}'$ is first-order definable in $\mathcal{CD}'$.
\end{theor}

We have already seen the proof of the easy(=only if) direction  of this theorem.
We prove the difficult direction in Section \ref{aspioerfdsc} by creating a model of $\mathcal{CD}'$ in $\mathcal{D}'$.

\begin{definition} A relation $R\subseteq \text{ob}(\mathcal{CD})^n$ is called {\it transposition invariant} if it is isomorphism invariant and $(G_1, \dots, G_n)\in R$ implies $(G_1^T,\dots, G_n^T)\in R.$
\end{definition}

\begin{kor}\label{cvbhjeuiw} A relation is first-order definable in $\mathcal{D}$ if and only if the corresponding isomorphism invariant relation of $\mathcal{CD}'$ is transposition invariant and first-order definable in $\mathcal{CD}'$.
\end{kor}

\begin{proof} The ``only if'' direction is obvious. For the ``if'' direction, let $R\subseteq \mathcal{D}^n$ be a relation that corresponds to a transposition invariant and first-order definable relation of $\mathcal{CD}'$. 
We need to show that $R$ is first-order definable in $\mathcal{D}$. 
We know, by Theorem \ref{9ahgfa}, that it is first-order definable in $\mathcal{D}'$. 
Let $\Phi(x_1, \dots, x_n)$ be a formula that defines it. 
Let $\Phi'(y,x_1, \dots, x_n)$ denote the formula that we get from $\Phi(x_1, \dots, x_n)$ by replacing the constant $A$ with $y$ at all of its occurrences. 
The set $\{A, A^T\}$ is easily defined (even without the usage of Theorem \ref{vxmknfsj}) in $\mathcal{D}$. 
Let us define
$$\Phi''(x_1, \dots, x_n):=\exists y (y\in\{A, A^T\} \wedge \Phi'(y,x_1, \dots, x_n)).$$
We claim that for $S:=\{(x_1, \dots, x_n) :\Phi''(x_1, \dots, x_n)\}$, $S=R$ holds. 
$R\subseteq S$ is clear as $\Phi'(A,x_1, \dots, x_n)$ defines $R$. Let $s\in S$. 
If this particular tuple $s$ is defined with $y=A$ in $\Phi''$ then $s\in R$ is obvious. 
If $s$ is defined with $y=A^T$ then $s^T$ can be defined with $y=A$ in $\Phi''$ and this yields $s^T\in R$, where the transpose is taken componentwise. Finally, the transposition invariance of $R$ implies $s\in R$.
 \end{proof}


We have already seen that in the first-order language of $\mathcal{CD}'$ we have access to the first-order language of digraphs. 
Let $G=(V,E)$ be an arbitrary fixed digraph with $V=\{v_1, \dots, v_n\}$. Then the formula 
\begin{multline}\exists x_1 \dots \exists x_n \forall y \bigg(
\bigwedge_{1\leq i \neq j \leq n }x_i \neq x_j \;\; \wedge 
\bigvee_{i=1}^n y=x_i \;\; \wedge \\ \nonumber
\bigwedge_{(v_i,v_j)\in E}(x_i, x_j)\in E  \;\; \wedge 
\bigwedge_{(v_i,v_j)\notin E} (x_i, x_j)\notin E  \bigg)
\end{multline}
defines $G$ in the first-order language of digraphs.
This leads to the following corollary of Theorem \ref{9ahgfa}.

\begin{kor}\label{allsppvnbt} In $\mathcal{D}'$, all elements are first-order definable. \qed
\end{kor}

\begin{kor}[=Theorem \ref{vxmknfsj}]\label{vncxbfui8} For all $G\in \mathcal{D}$, the set $\{G,G^T\}$ is first-order definable in $(\mathcal{D}, \leq)$. 
\end{kor}

\begin{proof} The proof goes with basically the same argument as we have seen in the proof of Corollary \ref{cvbhjeuiw}.
 \end{proof}

The previous two statements will only earn the ``title'' corollary truly, if we prove Theorem \ref{9ahgfa} without using them, which will be one way to approach the proof of Theorem \ref{9ahgfa}.

In the second-order language of digraphs---which has turned out to be available in the first-order language of $\mathcal{CD'}$---the formula
$$\exists H\subseteq G(\exists v,w\in G(v\in H \wedge w\notin H)\; \wedge \; \forall x, y \in G(x\to y \Rightarrow (x,y\in H \; \vee\; x,y \notin H))) $$
defines the set of not weakly connected digraphs. This means that the set of weakly connected digraphs is first-order definable in $\mathcal{D}$, by Corollary \ref{cvbhjeuiw}. That fact seems quite nontrivial to prove without Theorem \ref{9ahgfa}. This definability is surprising as the set of weakly connected digraphs is not definable in the first-order language of digraphs (by a standard model-theoretic argument). 


\section{Some notations and definitions needed from \cite{Kunos2015}}

In this section we recall additional notations and definitions from \cite{Kunos2015} that will be needed.

\begin{definition}\label{349587897819} For digraphs $G, G'\in \mathcal{D}$, let $G\mathrel{\dot{\cup}} G'$ denote their disjoint union, as usual.
\end{definition}

\begin{definition}\label{defE_n} Let ${E_n}$ $(n=1,2,\dots )$ denote the ``empty'' digraph with $n$ vertices and ${F_n}$ $(n=1,2,\dots )$ denote the ``full'' digraph with $n$ vertices: $$V(E_n)=\{v_1,v_2,\dots, v_n\},\;\;  E(E_n)=\emptyset,$$
$$V(F_n)=\{v_1,v_2,\dots, v_n\}, \;\; E(F_n)=V(F_n)^2.$$
\end{definition}

\begin{definition}\label{defIOL}
Let $I_n$, $O_n$, $L_n$ $(n=2,3,\dots)$ be the following (fig. \ref{O_n abra}.) digraphs: 
\begin{displaymath}
V(I_n)=V(O_n)=V(L_n)=\{v_1,v_2,\dots, v_n\}, 
\end{displaymath}
\begin{displaymath}
E(I_n)=\{(v_1,v_2),(v_2,v_3),\dots, (v_{n-1}, v_n) \},
\end{displaymath}
\begin{displaymath}
E(O_n)=\{(v_1,v_2),(v_2,v_3),\dots, (v_{n-1}, v_n), (v_n, v_1)\},
\end{displaymath}
\begin{displaymath}
E(L_n)=\{(v_1,v_1),(v_2,v_2),\dots, (v_n, v_n)\}.
\end{displaymath}
\end{definition}

\begin{figure}[h]
\begin{center}
\begin{tikzpicture}[line cap=round,line join=round,>=triangle 45,x=1.0cm,y=1.0cm]
\clip(2.22,0.82) rectangle (13.1,5.88);
\draw [->] (3.38,1.82) -- (3.4,2.68);
\draw [->] (3.4,2.68) -- (3.38,3.58);
\draw [->] (3.38,3.58) -- (3.36,4.54);
\draw [->] (3.36,4.54) -- (3.34,5.42);
\draw [->] (5.96,1.96) -- (5.36,3.54);
\draw [->] (5.36,3.54) -- (6,5);
\draw [->] (6,5) -- (7.8,4.98);
\draw [->] (7.8,4.98) -- (8.52,3.58);
\draw [->] (8.52,3.58) -- (8,2);
\draw [->] (8,2) -- (5.96,1.96);
\draw (3.18,1.28) node[anchor=north west] {$I_5$};
\draw (6.65,1.27) node[anchor=north west] {$O_6$};
\draw (11.1,1.3) node[anchor=north west] {$L_6$};
\begin{scriptsize}
\fill [color=black] (3.38,1.82) circle (1.5pt);
\draw[color=black] (3.52,2.1);
\fill [color=black] (3.4,2.68) circle (1.5pt);
\draw[color=black] (3.56,2.96);
\fill [color=black] (3.38,3.58) circle (1.5pt);
\draw[color=black] (3.54,3.86);
\fill [color=black] (3.36,4.54) circle (1.5pt);
\draw[color=black] (3.52,4.82);
\fill [color=black] (3.34,5.42) circle (1.5pt);
\draw[color=black] (3.5,5.7);
\fill [color=black] (5.96,1.96) circle (1.5pt);
\draw[color=black] (6.1,2.24);
\fill [color=black] (5.36,3.54) circle (1.5pt);
\draw[color=black] (5.52,3.82);
\fill [color=black] (6,5) circle (1.5pt);
\draw[color=black] (6.16,5.28);
\fill [color=black] (7.8,4.98) circle (1.5pt);
\draw[color=black] (7.9,5.26);
\fill [color=black] (8.52,3.58) circle (1.5pt);
\draw[color=black] (8.66,3.86);
\fill [color=black] (8,2) circle (1.5pt);
\draw[color=black] (8.16,2.28);
\fill [color=black] (10.56,2.02) circle (1.5pt);
\draw[color=black] (10.7,2.3);
\fill [color=black] (12.1,2.04) circle (1.5pt);
\draw[color=black] (12.26,2.32);
\fill [color=black] (10.56,3.48) circle (1.5pt);
\draw[color=black] (10.72,3.76);
\fill [color=black] (12.2,3.5) circle (1.5pt);
\draw[color=black] (12.36,3.78);
\fill [color=black] (10.52,4.78) circle (1.5pt);
\draw[color=black] (10.68,5.06);
\fill [color=black] (12.14,4.82) circle (1.5pt);
\draw[color=black] (12.3,5.1);
\draw[rotate around={-90:(10.56,2.02)}] [->] (10.56,2.02) arc (360:10:5pt);
\draw[rotate around={-90:(12.1,2.04)}] [->] (12.1,2.04) arc (360:10:5pt);
\draw[rotate around={-90:(10.56,3.48)}] [->] (10.56,3.48) arc (360:10:5pt);
\draw[rotate around={-90:(12.2,3.5)}] [->] (12.2,3.5) arc (360:10:5pt);
\draw[rotate around={-90:(10.52,4.78)}] [->] (10.52,4.78) arc (360:10:5pt);
\draw[rotate around={-90:(12.14,4.82)}] [->] (12.14,4.82) arc (360:10:5pt);
\end{scriptsize}
\end{tikzpicture}
\caption{$I_5$, $O_6$, $I_6$}
\label{O_n abra}
\end{center}
\end{figure}
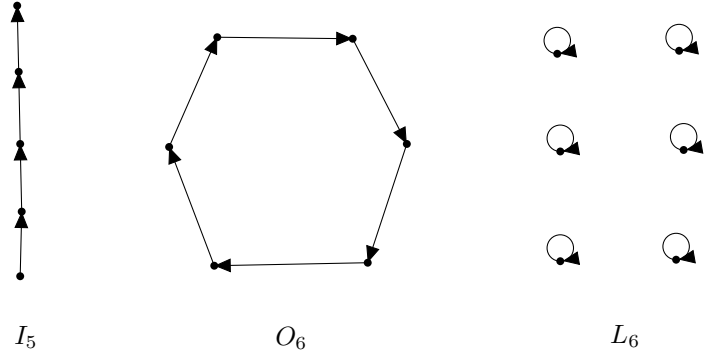

\begin{definition}\label{defOnto} Let $\mathcal{O}_{n}^{\to}$ denote the set of digraphs $X$ which we get by adding an edge that is not a loop to $O_n$.
\end{definition}

Note that $X\succ O_n$ for all $X\in \mathcal{O}_{n}^{\to}$.

\begin{definition}\label{defLfuggveny} For $G\in \mathcal{D}$, let 
$L(G)$ denote the digraph that we get from $G$ by adding all loops possible.
For $\mathcal{G}\subseteq \mathcal{D}$,
let us define $\mathcal{L}(\mathcal{G})=\{L(G) : G\in \mathcal{G}\}$.
\end{definition}

We would like to mention that this definition was a little different in \cite{Kunos2015}. We then assumed that $G$ has no loops which we do not do here.

\begin{definition}\label{defMfuggveny} For $G\in \mathcal{D}$, let $M(G)$ the digraph that we get from $G$ be by leaving all the loops out.
For $\mathcal{G}\subseteq \mathcal{D}$, let us define $\mathcal{M}(\mathcal{G}):=\{M(G):G\in \mathcal{G}\}.$
\end{definition}

\begin{definition}\label{defonl} Let $O_{n,L}$ be the following digraph: $V(O_{n,L})=\{v_1,v_2,\dots, v_n\}$, 
$E(O_{n,L})=E(O_n)\cup\{(v_1,v_1)\}$, meaning
\begin{displaymath}
E(O_{n,L})=\{(v_1,v_1),(v_1,v_2),(v_2,v_3),\dots, (v_{n-1}, v_n), (v_n, v_1)\}.
\end{displaymath}
\end{definition}

\begin{figure}[h]
\begin{center}
\begin{tikzpicture}[line cap=round,line join=round,>=triangle 45,x=1.0cm,y=1.2cm]
\clip(2.96,1.82) rectangle (5.18,3.32);
\draw [->] (3.34,2.02) -- (4.84,2.02);
\draw [->] (4.84,2.02) -- (4,3);
\draw [->] (4,3) -- (3.34,2.02);
\begin{scriptsize}
\fill [color=black] (3.34,2.02) circle (1.5pt);
\draw[color=black] (3.48,2.3);
\fill [color=black] (4.84,2.02);
\draw[color=black] (5,2.3);
\fill [color=black] (4,3) circle (1.5pt);
\draw[color=black] (4.16,3.28);
\draw[rotate around={-90:(4,3)}] [->] (4,3) arc (360:10:5pt);
\end{scriptsize}
\end{tikzpicture}
\caption{$O_{3,L}$}
\end{center}
\end{figure}

\begin{definition}\label{deffiugraf} Let $\male_n$ be the digraph with $n+1$ vertices $V(\male_n)=\{v_1,\dots , v_{n+1}\}$ for which
$v_1$, $v_2$, \dots, $v_{n}$ constitute a circle $O_n$ and the only additional edge in $\male_n$ is $(v_n,v_{n+1})$. 
Let $\male_n^L$ be the previous digraph plus one loop:
\begin{displaymath}
E(\male_n^L)=E(\male_n )\cup \{(v_{n+1},v_{n+1})\}.
\end{displaymath}
\end{definition}

\begin{figure}
\begin{center}
\begin{tikzpicture}[line cap=round,line join=round,>=triangle 45,x=1.0cm,y=1.0cm]
\clip(1.8,0.47) rectangle (10.55,5.54);
\draw [->] (3.14,1.44) -- (4.7,1.42);
\draw [->] (4.7,1.42) -- (5.47,2.66);
\draw [->] (5.47,2.66) -- (4.68,3.9);
\draw [->] (4.68,3.9) -- (3.18,3.88);
\draw [->] (3.18,3.88) -- (2.37,2.57);
\draw [->] (2.37,2.57) -- (3.14,1.44);
\draw [->] (4.68,3.9) -- (5.56,5.08);
\draw [->] (7.46,1.46) -- (9.02,1.44);
\draw [->] (9.02,1.44) -- (9.79,2.68);
\draw [->] (9.79,2.68) -- (9,3.92);
\draw [->] (9,3.92) -- (7.5,3.9);
\draw [->] (7.5,3.9) -- (6.69,2.65);
\draw [->] (6.69,2.65) -- (7.46,1.46);
\draw [->] (9,3.92) -- (9.88,5.1);
\draw (3.68,0.99) node[anchor=north west] {$\male_6$};
\draw (8.07,1.05) node[anchor=north west] {$\male_6^L$};
\begin{scriptsize}
\fill [color=black] (3.14,1.44) circle (1.5pt);
\draw[color=black] (3.28,1.73) ;
\fill [color=black] (4.7,1.42) circle (1.5pt);
\draw[color=black] (4.86,1.71) ;
\fill [color=black] (5.47,2.66) circle (1.5pt);
\draw[color=black] (5.62,2.94) ;
\fill [color=black] (4.68,3.9) circle (1.5pt);
\draw[color=black] (4.84,4.18) ;
\fill [color=black] (3.18,3.88) circle (1.5pt);
\draw[color=black] (3.34,4.16) ;
\fill [color=black] (2.37,2.57) circle (1.5pt);
\draw[color=black] (2.5,2.84) ;
\fill [color=black] (5.56,5.08) circle (1.5pt);
\draw[color=black] (5.72,5.36) ;
\fill [color=black] (7.46,1.46) circle (1.5pt);
\draw[color=black] (7.73,1.75) ;
\fill [color=black] (9.02,1.44) circle (1.5pt);
\draw[color=black] (9.31,1.73) ;
\fill [color=black] (9.79,2.68) circle (1.5pt);
\draw[color=black] (10.09,2.96) ;
\fill [color=black] (9,3.92) circle (1.5pt);
\draw[color=black] (9.29,4.2) ;
\fill [color=black] (7.5,3.9) circle (1.5pt);
\draw[color=black] (7.79,4.18) ;
\fill [color=black] (6.69,2.65) circle (1.5pt);
\draw[color=black] (6.97,2.92) ;
\fill [color=black] (9.88,5.1) circle (1.5pt);
\draw[color=black] (10.19,5.38) ;
\draw[rotate around={-90:(9.88,5.1)}] [->] (9.88,5.1) arc (360:10:5pt);
\end{scriptsize}
\end{tikzpicture}
\caption{$\male_6$ and $\male_6^L$}
\end{center}
\end{figure}

\section{The proof of the main theorem (Theorem \ref{9ahgfa})}\label{aspioerfdsc}

In this chapter we prove  the ``if'' direction of Theorem \ref{9ahgfa}. 
Here, if we just write ``definability'', we will always mean first-order definability in $\mathcal{D}'$.\\

In the proof we discuss in this section, the statement of Corollary \ref{allsppvnbt} turns out to be very useful as there are a number of specific digraphs whose definability is used throughout our proofs. 
There are two different approaches to the proof of this section according to our intentions with the main result of \cite{Kunos2015}, that is Theorem \ref{vxmknfsj}. \\
EITHER
\begin{itemize}
\item We use Corollary \ref{allsppvnbt}, considering it as a consequence of Theorem \ref{vxmknfsj}, see the proof of \cite[Theorem 3.3]{Kunos2015}. 
In this case, we use paper \cite{Kunos2015}, therefore its result cannot be considered as a corollary of Theorem \ref{9ahgfa}.
\end{itemize}
OR
\begin{itemize}
\item We use the following lemma to replace the statement of Corrolary \ref{allsppvnbt} in the special cases of those specific digraphs that we would use the statement of  Corrolary~\ref{allsppvnbt} for. 
This way Corrolary \ref{allsppvnbt} and the main result of \cite{Kunos2015} can both be viewed as corrolaries of Theorem \ref{9ahgfa}.
\end{itemize}

The latter approach requires the following lemma.

\begin{lemm}\label{vbncxmze} The following digraphs (of at most 9 elements) are first-order definable in $\mathcal{D'}$: $I_2$, $L_1$, $E_2$, $A$, $A^T$, and the digraphs under (\ref{hxgetrirow}), (\ref{ncbhzuippqw}), and (\ref{cnjusiuduh}).
\end{lemm}

\begin{proof} The proof of this lemma must go without the usage of Corrolary \ref{allsppvnbt} (and Theorem \ref{vxmknfsj}). 
We only need to consider some (finite) levels at the ``bottom'' of the poset $\mathcal{D}$. This means it is only a matter of time for someone to create this proof. The detailed proof would be technical and it would bring nothing new to the table, so we skip it.
 \end{proof}

From now on, either of the two approaches above can be followed---the proof in the remainder of this chapter is the same in both cases. It is up to the reader which approach he favors and has in mind while reading the rest of the paper.

\begin{lemm}
The sets
$\mathcal{E}:=\{E_n : n\in \mathbb{N}\}$, $\mathcal{L}:=\{L_n:n\in \mathbb{N}\}$ and the relation $\{(L_n,E_n): n\in\mathbb{N}\}$ are definable.
\end{lemm}

\begin{proof} $\mathcal{E}$ is the set of $X\in \mathcal{D}$ for which $I_2\nleq X$ and $L_1\nleq X$.
$\mathcal{L}$ is the set of those digraphs $X\in \mathcal{D}$ for which there exists $E_i\in\mathcal{E}$ such that $X$ is maximal with the properties $E_i\leq X$, $E_{i+1}\nleq X$ and $I_2\nleq X$. ($E_{i+1}$ is easily defined using $E_i$ as it is the only cover of $E_i$ in the set $\mathcal{E}$.) \\
The relation consists of those pairs $(X,Y)\in\mathcal{D}^2$ for which $X\in \mathcal{L}$, and $Y$ is maximal element of $\mathcal{E}$ that is embeddable into $X$. 
 \end{proof}

The relations
\begin{equation}\label{yiteja}
\{(G,E_n):E_n\leq G, \;E_{n+1}\nleq G\},\; \text{ and}
\end{equation}
\begin{equation}\label{leitaf}
\{(G,L_n):L_n\leq G, \;L_{n+1}\nleq G\}
\end{equation}
are obviously definable, from which the following relations are definable too:

\begin{definition}\label{xmvbyxhu8231}
$$ \mathfrak{E}:=\{(G, K): \exists \; E_n\in\mathcal{E},\text{\;\;for which\;\;} (G,E_n),(K,E_n)\in (\ref{yiteja})\}, $$
$$ \mathfrak{L}:=\{(G,K): \exists \; L_n\in\mathcal{L},\text{\;\;for which\;\;} (G,L_n),(K,L_n)\in (\ref{leitaf})\}. $$
\end{definition}

\begin{definition}\label{123897834611} Let $\mathcal{O}$ denote the set of those digraphs that are disjoint unions of circles ($O_n$ for $n\geq 2$) of not necessarily different sizes.
\end{definition}

\begin{lemm} $\mathcal{O}$ is definable. 
\end{lemm}

\begin{proof} Let $\mathcal{H}$ be the set consisting of those $X\in\mathcal{D}$ for which there exists $E_n\in\mathcal{E}$ such that  $X$  is maximal with the properties
\begin{equation}\label{aiuueelllbv}
E_2\leq X, \; A\nleq X, \; A^T\nleq X, \; L_1\nleq X, \; \text{and}\;\; (X,E_n)\in\mathfrak{E}. 
\end{equation}
We state that 
\begin{equation}\label{karuervcy}
\mathcal{H}=\mathcal{O}\cup \{G\mathrel{\dot{\cup}} E_1: G\in \mathcal{O}\}.
\end{equation}

Let $G\in \mathcal{H}$. It is easy to see that there can be at most 1 weakly connected component of $G$ that has only 1 vertex (and hence is isomorphic to $E_1$) as the opposite would conflict the maximality of $G$. 
The conditions $A\nleq X$ and $A^T\nleq X$ mean there is no vertex in $G$ that is either an ending or a starting point of two separate edges, respectively. Therefore every weakly connected component of $G$ is either a circle or only one element.
Finally, $\mathcal{O}$ is the set of $X\in \mathcal{D}$ for which $X\in\mathcal{H}$ but there is no such $Y\in \mathcal{H}$ that $Y\prec X$.
 \end{proof}

\begin{lemm}\label{ppghhfjasd} The following sets and relations are definable:
$$\mathcal{O}_{\cup}:=\{O_n:n\geq2\}, \;\;\; \{(O_n,E_n): n\geq 2\},$$
\begin{equation}\label{yxcnyxbmieuwqieo99}
\{F_n:n\in\mathbb{N}\}, \;\;\; \{(F_n,E_n): n\in\mathbb{N}\},
\end{equation}
\begin{equation}\label{adubvcipqw} \{(G,M(G)): G\in\mathcal{D}\},
\end{equation}
$$\mathfrak{M}:=\{(X,Y): \exists Z ((X,Z),(Y,Z)\in (\ref{adubvcipqw}))\},$$
\begin{equation}\label{cvjnhdufasid}
\{(G,L(G)): G\in\mathcal{D}\}.
\end{equation}
\end{lemm}

\begin{proof} $\mathcal{O}_{\cup}$ is the set of digraphs $X\in \mathcal{D}$ for which $X\in \mathcal{O}$ but there is no $Y\in \mathcal{O}$ such that $Y<X$. 
The corresponding relation $\{(O_n,E_n): n\geq 2\}$ is definable with \eqref{yiteja}. \\
The set under (\ref{yxcnyxbmieuwqieo99}) consists of those $X\in\mathcal{D}$ for which $X<Y$ implies $(X,Y)\notin \mathfrak{E}$. 
The corresponding relation is defined as above.\\
(\ref{adubvcipqw}) is the set of pairs  $(X,Y)\in \mathcal{D}^2$ for which $Y$ is maximal with the conditions $Y\leq X$ and $L_1\nleq Y$. \\
$\mathfrak{M}$ is already given by a first-order definition. \\
(\ref{cvjnhdufasid}) is the set of pairs  $(X,Y)\in \mathcal{D}^2$ for which $Y$ is maximal with the property that $(X,Y)\in \mathfrak{M}$.
 \end{proof}

\begin{lemm}\label{238971875642} The following relation is definable:
$$\mathfrak{E}_+:=\{(E_n,E_m,E_{n+m}): n,m\in\mathbb{N}\}.$$
\end{lemm}

\begin{proof} The relation $\mathfrak{E}_+$ consists of the triples $(X,Y,Z)\in\mathcal{D}^3$ that satisfy the following conditions. 
$X,Y\in\mathcal{E}$, meaning $X=E_i$ and $Y=E_j$ for some $i,j\in\mathbb{N}$. 
With Lemma \ref{ppghhfjasd}, $M(F_j)$ can be defined (with $E_j$). 
Let $F_j^*$ denote the digraph the we get from $M(F_j)$ by adding one loop. 
This is the only digraph $W\in\mathcal{D}$ for which $M(F_j)\prec W$ and $L_1\leq W$. 
Now the digraph $L_i  \mathrel{\dot{\cup}} M(F_j)$ is definable as the digraph $Q\in \mathcal{D}$ which is minimal with the conditions $L_i\leq Q$, $M(F_j)\leq Q$ and $F_j^*\nleq Q$. Finally, $Z\in \mathcal{E}$ such that $(Z,L_i  \mathrel{\dot{\cup}} M(F_j))\in \mathfrak{E}$.
 \end{proof}

\begin{lemm} The following relation is definable:
\begin{equation}\label{hatluevp}
\{(E_n,E_m):1\leq n<m\leq 2n\}.
\end{equation}
\end{lemm}

\begin{proof} The relation is the set of those pairs $(X,Y)\in\mathcal{D}^2$ which satisfy the following conditions. For $X\in \mathcal{E}$, meaning $X=E_n$, we can define  
$E_{2n}$ to be the element from the set $\mathcal{E}$ for which $(E_n,E_n,E_{2n})\in\mathfrak{E}_+$. Finally, $Y\in \mathcal{E}$ and $E_n<Y\leq E_{2n}$.
 \end{proof}

\begin{lemm}\label{983489887772543214} Let $O_n^*:=O_{n+1} \mathrel{\dot{\cup}} O_{n+2} \mathrel{\dot{\cup}} \dots \mathrel{\dot{\cup}} O_{2n}$. The relation
\begin{equation} \label{gmcuyalu}
\{(O_n^*,E_n): n\in\mathbb{N}\}
\end{equation}
and the set $\{O_n^*:n\in\mathbb{N}\}$
are definable. 
\end{lemm}

\begin{proof} The relation (\ref{gmcuyalu}) can be defined as the set of pairs $(X,Y)\in \mathcal{D}^2$ satisfying the following conditions.
$Y\in \mathcal{E}$, meaning $Y=E_n$. $X$ satisfies $X\in \mathcal{O}$ and is minimal with the following property: for all $O_i \in \mathcal{O}_{\cup}$ for which $(E_n,E_i)\in (\ref{hatluevp})$ holds, $O_i\leq X$. \\ 
With the relation (\ref{gmcuyalu}), the set is easily defined the usual way.
 \end{proof}

\begin{lemm} The following relation is definable:
$$\{(X,E_n): \; 2\leq n, \;\; X\in\mathcal{O}_n^{\to}\}.  $$
\end{lemm}

\begin{proof} The relation consists of those pairs $(X,Y)\in\mathcal{D}^2$ that satisty the following conditions. 
$Y\in \mathcal{E}$ and $E_2\leq Y$, meaning $Y=E_n$, where $2\leq n$. 
$O_n\prec X$, $L_1\nleq X$, and $(X,O_n)\in\mathfrak{E}$.
 \end{proof}

\begin{definition}\label{13541235} Let $1<i,j$ be integers and let us consider the circles $O_i, O_j$ and $E_1$ with
$$ V(O_i)=\{v_1,\dots,v_i\},\;\; V(O_j)=\{v^1,\dots, v^j\}, \;\; V(E_1)=\{u\}.$$
Let  $\male_{i,j}^L$ denote the following digraph:
$$V(\male_{i,j}^L):=V(O_i)\cup V(O_j) \cup V(E_1), \;\; E(\male_{i,j}^L):=E(O_i)\cup E(O_j)\cup \{(v_1,u),(v^1,u),(u,u)\}.$$
\end{definition}

\begin{lemm}\label{apoowjsbc} 
 Let 
$$O_{n,L}^*:=O_{n+1,L} \mathrel{\dot{\cup}} O_{n+2,L} \mathrel{\dot{\cup}} \dots \mathrel{\dot{\cup}} O_{2n,L}.$$

The following sets and relations are definable:
\begin{equation}\label{telgapnn}
\{(\male_n,E_n): n\geq2\}, \;\;\; \{\male_n: n\geq2\},
\end{equation}
\begin{equation}\label{bahhstz}
\{(\male_n^L, E_n): n\geq2\}, \;\; \{\male_n^L: n\geq2\},
\end{equation}
\begin{equation}\label{vlapkaaa}
\{(\male_{i,j}^L,E_i,E_j):1<i,j,\; i\neq j\}, \;\; \{\male_{i,j}^L:1<i,j,\; i\neq j\},
\end{equation}
\begin{equation}\label{98782133567236}
\{O_{n,L}: n\geq2\}, \;\;\; \{(O_{n,L},E_n):n\geq2 \},
\end{equation}
\begin{equation}\label{745673682987123}
\{O_{n,L}^*:n\in\mathbb{N}\}, \;\;\; \{(O_{n,L}^*,E_n):n\in\mathbb{N} \}.
\end{equation}
\end{lemm}

\begin{proof} The relation (\ref{telgapnn}) consists of those pairs $(X,Y)\in \mathcal{D}^2$ that satisfy the following. 
$Y\in \mathcal{E}$, meaning $Y=E_n$. 
There exists $Z\in \mathcal{D}$ for which $O_n\prec Z\prec X$, $(E_{n+1}, X)\in \mathfrak{E}$, and $L_1\nleq X$. There exists no $Z\in \mathcal{O}_n^{\to}$ for which $Z\leq X$. Finally, $A^T\leq X$. The corresponding set is easily defined using the relation we just defined. \\
The set under (\ref{98782133567236}) consists of those digraphs $X\in \mathcal{D}$ for which there exists $O_n\in\mathcal{O}_{\cup}$ such that $O_n\prec X$ and $L_1\leq X$. 
The corresponding relation is easily defined. \\
The relation (\ref{bahhstz}) consists of those pairs $(X,Y)\in \mathcal{D}^2$ that satisfy the following. $Y\in \mathcal{E}$, meaning $Y=E_n$. With the relation (\ref{telgapnn}), $\male_n$ is definable. Now $X$ is determined by the following properties: $\male_n\prec X$, $L_1\leq X$ and $O_{n,L}\nleq X$. 
The corresponding set is easily defined using the relation we just defined. \\
The relation (\ref{vlapkaaa}) consists of those triples $(X,Y,Z)\in\mathcal{D}^3$ that satisfy the following. $Y,Z\in\mathcal{E}$ such that $E_2\leq Y,Z$ and $Y\neq Z$, meaning $Y=E_i$, $Z=E_j$ for some $1<i,j$, $i\neq j$. 
Now $O_i \mathrel{\dot{\cup}} O_j$ is the digraph $W\in\mathcal{D}$ determined by $W\in \mathcal{O}$, $(W,E_{i+j})\in\mathfrak{E}$ and $O_i, O_j\leq  W$. 
$O_i \mathrel{\dot{\cup}} O_j \mathrel{\dot{\cup}} E_1$  is the digraph $W$ determined by $O_i \mathrel{\dot{\cup}} O_j \prec W$ and $(O_i \mathrel{\dot{\cup}} O_j, W)\notin \mathfrak{E}$. 
Finally, $X$ is defined by:
$$\exists W_1,W_2: O_i \mathrel{\dot{\cup}} O_j \mathrel{\dot{\cup}} E_1 \prec W_1\prec W_2\prec X, \;\; (X,O_i \mathrel{\dot{\cup}} O_j \mathrel{\dot{\cup}} E_1) \in \mathfrak{E},$$
$$ L_1\leq X, \;\; O_{i,L}\nleq X, \;\; O_{j,L}\nleq X, \;\; \male_i^L\leq X, \;\; \male_j^L\leq X.$$
The corresponding set is easily defined using the relation we just defined.\\
The relation $\{(O_{n,L}^*,E_n):n\in\mathbb{N} \}$ consists of those pairs $(X,Y)\in \mathcal{D}^2$ that satisfy the following conditions. $Y\in\mathcal{E}$, meaning $Y=E_n$. For $X$, the following properties hold:
\begin{itemize}
\item $O_n^*\leq X$ and $(X,O_n^*)\in\mathfrak{E}$,
\item  $O_i\leq O_n^* \Rightarrow ( \; Y\in \mathcal{O}_i^{\to} \Rightarrow Y\nleq X),$
\item  $O_i\leq O_n^* \Rightarrow \male_i\nleq X,$
\item $O_i\leq O_n^* \Rightarrow O_{i,L}\leq X,$
\item $L_{n+1}\nleq X.$
\end{itemize}
With the relation we just defined the corresponding set is easily defined.
 \end{proof}

\begin{definition}\label{756728767198470} Let us denote the vertices of $O_i$ and $O_j$ with
$$ V(O_i)=\{v_1,\dots,v_i\},\;\; V(O_j)=\{v^1,\dots, v^j\}. $$
Let $O_{i\to j}$ denote the digraph 
$$V(O_{i\to j})=V(O_i)\cup V(O_j), \;\;\; E(O_{i\to j})=E(O_i)\cup E(O_j) \cup \{(v_1,  v^1)\}.$$
\end{definition}

\begin{lemm} The following relation and set are definable:
\begin{equation}\label{haywskz}
\{(O_{i\to j}, E_i, E_j): i,j\geq2\}, \;\;\; \{O_{i\to j}: i,j\geq2\}
\end{equation}
\end{lemm}

\begin{proof} The relation (\ref{haywskz}) consists of those triples $(X,Y,Z)\in \mathcal{D}^3$ for which the following conditions hold. $Y,Z\in \mathcal{E}$ satisfy $E_2\leq Y,Z$, meaning $Y=E_i$ and $Z=E_j$, where $i,j\geq2$. $(X,E_{i+j})\in \mathfrak{E}$ and $W\prec X$, where $W\in \mathcal{O}$ is such that $(W,X)\in \mathfrak{E}$ and precisely $O_i$ and $O_j$ are embeddabe into $W$ from the set $\mathcal{O}_{\cup}$ (here $i=j$ is possible). Finally, $\male_i\leq X$. The set is easily defined using the relation. 
 \end{proof}

The proof of the crucial Lemma \ref{fiduhweori} requires a lot of nontrivial preparation which we begin here.

\begin{definition}\label{7358712837283791} Let $\mathcal{W}(G)$ denote the set of weakly connected components of $G$. 
\end{definition}

\begin{definition}\label{7538761239874} Let
$$G \sqsubseteq G' \Leftrightarrow M(G)\leq M(G'), \; \text{ and }\;  G \sqsubset G' \Leftrightarrow M(G)<M(G'),$$
$$G\equiv G'  \; \Leftrightarrow \; M(G)=M(G') \; (\Leftrightarrow \; (G,G')\in\mathfrak{M}), \text{ that is } \equiv \;\; = \;\; \sqsubseteq \cap \sqsubseteq^{-1},$$
$$\equiv^C_G\;\; :=\{H\in \mathcal{W}(G): H\equiv C\}.$$
\end{definition}

Let us use the abbreviation {\it wcc}=``weakly connected component'' and {\it wccs} for the plural.

$\sqsubseteq$ is obviosly a quasiorder and
$\equiv^C_G$ is the set of the wccs of $G$ that are equivalent to $C$ with respect to the equivalence  $\equiv$.

We say that a wcc $W$ of $G$ is {\it raised} by the embedding  $\varphi: G\to G'$ if for the wcc $W'$ of $G'$ that it embeds into, i. e. $\varphi(W)\subseteq W'$, $W\sqsubset W'$  holds. In this case, we say that {\it$W$ is raised into $W'$}. 
A wcc $W$ of $G$ is either raised or embeds into $\equiv^W_{G'}$ (considered now as a subgraph of $G'$).

\begin{lemm}\label{nvbdfuzer} Let $G$ and $G'$ be digraphs having $n$ vertices such that $G\equiv G'$. 
Let $\varphi$ be an embedding $G\to G' \mathrel{\dot{\cup}} O_n^*$.
Let us suppose that $W$ and $W'$ are wccs of $G$ and $G'$ respectively, such that  $W$ is raised into $W'$.
Then $W'\equiv I_m$ for some $m$, and consequently $W\equiv I_{m'}$ for some $m'< m$.
\end{lemm}

\begin{proof} It suffices to show that $M(W')$ can be embedded into $O_n^*$, that is what we are going to do. 
For an arbitrary wcc $V$ of $G$, it is clear that $\equiv^V_G$ and $\equiv^V_{G'}$ are either bijective under $\varphi$ (considered as subsgraphs of $G$ and $G'$) or a wcc of $\equiv^V_G$ is raised. 
The fact that $W$ is raised into $W'$ excludes $\equiv^{W'}_G$ and $\equiv^{W'}_{G'}$ being bijective as these two subgraphs are $\equiv$--equivalent, so a bijection would only be possible if only $\equiv^{W'}_G$ was mapped into $\equiv^{W'}_{G'}$. 
This means that a wcc $W_1$ of $\equiv^{W'}_G$ is raised into some wcc $W'_1$. 
If $W'_1$ is a wcc of $O_n^*$, then we are done as clearly
$$W\sqsubset W_1 \sqsubset W'_1.$$
If this is not the case, then we repeat the same argument to get wccs $W_2 \in  \equiv^{W'_1}_G$, and $W'_2$ such that $W_2$ is raised into $W'_2$. Again, if $W'_2$ is in $O_n^*$, then we are done as
$$W\sqsubset W_1 \sqsubset W_2 \sqsubset W'_2.$$
If not, we repeat the argument. Since an infinite chain of wccs with strictly increasing size is impossible, we will get to our claim eventually.
 \end{proof}

We are in the middle of the preparation for Lemma \ref{fiduhweori}. The following Lemma \ref{weiurzwe} is the key, the most difficult part of the paper. 
Before the lemma we give an example to aid the understanding of its statement.
We consider the digraphs $G$ and $G'\mathrel{\dot{\cup}} O_n^*$ and we are interested if the assumptions 
\begin{itemize}
\item $G\leq G'\mathrel{\dot{\cup}} O_n^*$,
\item $G\equiv G'$, and 
\item $G$ and $G'$ have the same number of loops
\end{itemize}
force $G=G'$? 
The answer is negative and a counterexample is shown in Figure \ref{sdfkjfhsufi}.
To prove Lemma \ref{fiduhweori} we will need to ensure that $G=G'$ with a first-order definition.
Observe the following. 
Let $\overline{G}$ denote the digraph  we get from $G$ by adding a loop to the vertex labeled with $v$. 
Now it is impossible to add one loop to $G'$ such that we get a $\overline{G'}$ for which $\overline{G}\leq \overline{G'}\mathrel{\dot{\cup}} O_3^*$ holds. 
We just showed the following property: we can add some loops to $G$, getting $\overline{G}$, such that it is impossible to add the same number of loops to $G'$, getting $\overline{G'}$, such that $\overline{G}\leq \overline{G'}\mathrel{\dot{\cup}} O_3^*$ holds. If we have $G=G'$  this property does not hold, obviosly.
Have we found a property that, together with the three above, ensures $G=G'$? The following lemma answers this question affirmatively. 
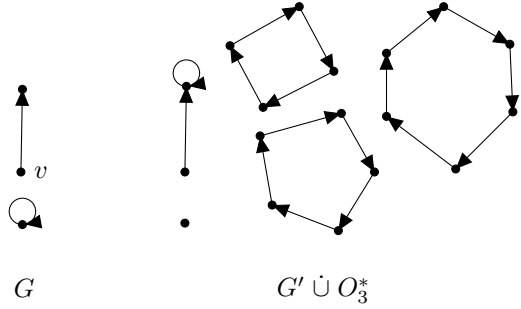
\begin{figure}[h]
\begin{center}
\begin{tikzpicture}[line cap=round,line join=round,>=triangle 45,x=1.0cm,y=1.0cm]
\clip(1.1599999999999995,0.73) rectangle (8.699999999999996,4.939999999999999);
\draw [->] (1.82,2.52) -- (1.84,3.62);
\draw [->] (4.0,2.52) -- (4.02,3.66);
\draw [->] (5.04,3.38) -- (4.6,4.2);
\draw [->] (4.6,4.2) -- (5.52,4.72);
\draw [->] (5.52,4.72) -- (5.98,3.86);
\draw [->] (5.98,3.86) -- (5.04,3.38);
\draw [->] (5.16,2.08) -- (5.0,3.0);
\draw [->] (5.0,3.0) -- (6.08,3.3);
\draw [->] (6.08,3.3) -- (6.52,2.52);
\draw [->] (6.52,2.52) -- (6.04,1.74);
\draw [->] (6.04,1.74) -- (5.16,2.08);
\draw [->] (6.68,3.26) -- (6.68,4.1);
\draw [->] (6.68,4.1) -- (7.44,4.72);
\draw [->] (7.44,4.72) -- (8.32,4.22);
\draw [->] (8.32,4.22) -- (8.36,3.32);
\draw [->] (8.36,3.32) -- (7.6,2.56);
\draw [->] (7.6,2.56) -- (6.68,3.26);
\draw (1.6,1.2199999999999995) node[anchor=north west] {$G$};
\draw (5.1,1.27) node[anchor=north west] {$G'\mathrel{\dot{\cup}} O_3^*$};
\draw (1.87,2.52) node[anchor= west] {$v$};
\draw[rotate around={-90:(1.84,1.82)}] [->] (1.84,1.82) arc (360:10:5pt);
\draw[rotate around={-90:(4.02,3.66)}] [->] (4.02,3.66) arc (360:10:5pt);
\begin{scriptsize}
\draw [fill=black] (1.82,2.52) circle (1.5pt);
\draw [fill=black] (1.84,3.62) circle (1.5pt);
\draw [fill=black] (4.0,2.52) circle (1.5pt);
\draw [fill=black] (4.02,3.66) circle (1.5pt);
\draw [fill=black] (1.84,1.82) circle (1.5pt);
\draw [fill=black] (4.0,1.84) circle (1.5pt);
\draw [fill=black] (5.04,3.38) circle (1.5pt);
\draw [fill=black] (4.6,4.2) circle (1.5pt);
\draw [fill=black] (5.52,4.72) circle (1.5pt);
\draw [fill=black] (5.98,3.86) circle (1.5pt);
\draw [fill=black] (5.16,2.08) circle (1.5pt);
\draw [fill=black] (5.0,3.0) circle (1.5pt);
\draw [fill=black] (6.08,3.3) circle (1.5pt);
\draw [fill=black] (6.52,2.52) circle (1.5pt);
\draw [fill=black] (6.04,1.74) circle (1.5pt);
\draw [fill=black] (6.68,3.26) circle (1.5pt);
\draw [fill=black] (6.68,4.1) circle (1.5pt);
\draw [fill=black] (7.44,4.72) circle (1.5pt);
\draw [fill=black] (8.32,4.22) circle (1.5pt);
\draw [fill=black] (8.36,3.32) circle (1.5pt);
\draw [fill=black] (7.6,2.56) circle (1.5pt);
\end{scriptsize}
\end{tikzpicture}
\end{center}
\caption{A $G$ and a corresponding $G'\mathrel{\dot{\cup}} O_3^*$ forming a counterexample}
\label{sdfkjfhsufi}
\end{figure}

\begin{lemm}\label{weiurzwe} Let $G,G'$ be digraphs with $n$ vertices and with the same number of loops. Let us suppose $G\equiv G'$ and $G\leq G'\mathrel{\dot{\cup}} O_n^*$. 
Then $G\neq G'$ holds if and only if we can add some loops to $G$ so that we get the digraph $\overline{G}$ such that it is impossible to add the same number of loops to $G'$, getting the digraph $\overline{G'}$, such that $\overline{G}\leq \overline{G'}\mathrel{\dot{\cup}} O_n^*$. In formulas this is: there exists a digraph $\overline{G}$ for which
$$ G\leq \overline{G}, \;\; G\equiv \overline{G}$$
such that there exists no digraph $X$ for which 
$$ G'\mathrel{\dot{\cup}} O_n^* \leq X, \;\; X\equiv G'\mathrel{\dot{\cup}} O_n^*, \;\; X\leq L(G)\mathrel{\dot{\cup}} O_n^*, \;\; (\overline{G}, X)\in\mathfrak{L}.$$
\end{lemm}

\begin{proof} 
The direction $\Leftarrow$ (or rather its contrapositive) is obvious. Accordingly, let us suppose $G\neq G'$.

Let $C$ denote the largest joint subgraph consisting of whole wccs of both $G$ and $G'$. 
Let us introduce the so-called {\it reduced subgraphs}:
\begin{equation}\label{cvxyxopp}
G=C\mathrel{\dot{\cup}} G_R, \text{ and } G'=C\mathrel{\dot{\cup}} G'_R.
\end{equation}
Observe that the digraphs $G_R$ and $G'_R$ are not empty and $G_R \equiv G'_R$.

Let $W$ denote a $\sqsubseteq$-maximal wcc of $G_R$.
We claim $W\equiv I_k$ for some $k>1$, and 
\begin{equation}\label{sdsjhjsdp}
|=_G^{I_k}|-|=_{G'}^{I_k}|=|\equiv_{G_R}^{I_k}|,
\end{equation}
or equivalently, all wccs of $\equiv_{G_R}^{I_k}$ are loop-free.
Let $\varphi$ be an embedding $G\to G'\mathrel{\dot{\cup}} O_n^*$. 
Observe that $\varphi$ raises a wcc isomorphic to $W$ as $G'$ has less wccs isomorphic to $W$ by the definitions of the reduced subgraphs. 
Hence, by Lemma \ref{nvbdfuzer}, we have $W\equiv I_k$ for some $k\geq 1$. 
This is less then what we claimed, the exclusion of the case $k=1$ remains to be seen yet. 
It is easy to see from the definitions that \eqref{sdsjhjsdp} is equivalent to the fact that all wccs of $\equiv_{G_R}^{I_k}$ are loop-free. 
Let us suppose, for contradiction, that a wcc $V$ of $\equiv_{G_R}^{I_k}$ has a loop in it. 
Observe that the loops of $G$ and $G'$ are bijective under $\varphi$. 
Moreover, from the maximality of $W$, it is easy to see that for a wcc $U\sqsupset I_k$ of $G$, the loops of $\equiv^U_G$ are bijective with the loops of $\equiv^U_{G'}$ under $\varphi$. 
Consequently, none of the wccs of $=^V_G$ is raised as, by our previous argument, there is no component to be raised into.
 Hence $|=^V_G|\leq |=^V_{G'}|$, which clearly contradicts the fact that $V$ is an element of $\equiv_{G_R}^{I_k}$. 
We have proven \eqref{sdsjhjsdp}, only the exclusion of $k=1$ remains from our claim above. Let us suppose $k=1$ for contradiction.  
An arbitrary wcc $K$ of $G$ is either $K\equiv I_1$ or $K\sqsupset I_1$. 
In the latter case, as we have seen above, the loops of $\equiv_G^K$ are bijective with $\equiv_{G'}^K$. 
If $K\equiv I_1$, then we have shown above that the nonempty set $\equiv_{G_R}^{I_1}$ is loop-free, while, as a consequence, $\equiv_{G'_R}^{I_1}$, that has the same number of elements, contains of loops. 
This means  $G$ has more loops then $G'$ does, a contradiction. 
We have entirely proven our claim. 

Observe that, from our claim above, the nonempty set $\equiv_{G'_R}^{I_k}$ contains no loop-free elements. 
Take $W' \in \; \equiv_{G'_R}^{I_k}$. 
We make the digraph $\overline{W}$ from $I_k$ by adding 1 loop so that $\overline{W}\neq W'$. 
This is possible because either $W'$ has loops on all of its vertices, then (using $k>1$) adding the loop arbitrarily suffices, or there is a vertex that has no loop on it, then adding the loop the this vertex in $I_k$ does.

Now we create the digraph $\overline{G}$ of the theorem by adding 1 loop to each loop-free wcc of $G$. 
To the wccs of $=^{I_k}_G$  we add 1 loop each such that they all become $\overline{W}$. 
To all other loop-free wccs of $G$, we add 1 loop each arbitrarily.

To prove that $\overline{G}$ is sufficient, we suppose, for contradiction, that, by adding the same number of loops to $G'$, we can get some $\overline{G'}$ for which $\overline{G}\leq \overline{G'}\mathrel{\dot{\cup}} O_n^*$. 
Let $\phi$ be an embedding $\overline{G}\to \overline{G'}\mathrel{\dot{\cup}} O_n^*$. For each wcc has a loop in $\overline{G}$, $\phi$ is technically an isomorphism $\phi: \overline{G}\to \overline{G'}$. Our final claim is, 
\begin{equation}\label{sdouw8q4u}
|=^{\overline{W}}_{\overline{G}}|>|=^{\overline{W}}_{\overline{G'}}|,
\end{equation}
which contradicts the existence of the isomorphism $\phi: \overline{G}\to \overline{G'}$. If \eqref{sdouw8q4u} gets proven, we are done. Using the decomposition \eqref{cvxyxopp} and the knowledge on how $\overline{G}$ was created, the left side of \eqref{sdouw8q4u} is
\begin{equation}\label{i99834udj}
|=^{\overline{W}}_{\overline{G}}|\;\;=\;\;
|=^{I_k}_G|+|=^{\overline{W}}_{G_R}|+|=^{\overline{W}}_{C}|\;\;=\;\;
|=^{I_k}_G|+|=^{\overline{W}}_{C}|,
\end{equation}
since $\equiv^{W}_{G_R} \;=\;\;  \equiv^{I_k}_{G_R}$ was shown to be loop-free above. 
Observe that even though we do not know exactly how $\overline{G'}$ was created, a component isomorphic to $\overline{W}$ can only appear in it if either it was already in $G'$ and no loop was added to that specific component, or the component was isomorphic to $I_k$ in $G'$, but a loop was added to the right place. This implies 
\begin{equation}\label{1234743}
|=^{\overline{W}}_{\overline{G'}}|\;\;\leq\;\;
|=^{I_k}_{G'}|+|=^{\overline{W}}_{G'_R}|+|=^{\overline{W}}_{C}|.
\end{equation}
Using \eqref{i99834udj} and \eqref{1234743}, it is enough to show that 
$$|=^{I_k}_G|+|=^{\overline{W}}_{C}|\;\;>\;\;
|=^{I_k}_{G'}|+|=^{\overline{W}}_{G'_R}|+|=^{\overline{W}}_{C}|,$$
or equivalently,
$$|=^{I_k}_G|-|=^{I_k}_{G'}|\;\;>\;\;
|=^{\overline{W}}_{G'_R}|.$$ 
Using \eqref{sdsjhjsdp}, this turns into $|\equiv_{G_R}^{I_k}| > |=^{\overline{W}}_{G'_R}|$, which is obvious considering how $\overline{W}$ was created. We have proven \eqref{sdouw8q4u}, we are done.
 \end{proof}

\begin{lemm}\label{fiduhweori} The following relation is definable:
\begin{equation}\label{yyaajjjqwe}
\{(G,G\mathrel{\dot{\cup}} O_n^*): G\in\mathcal{D}, \; |V(G)|=n\}.
\end{equation}
\end{lemm}

\begin{proof} The relation in question is the set of pairs $(X,Y)\in\mathcal{D}^2$ that satisfy the following conditions. 
Let $(X,E_n)\in\mathfrak{E}$. 
Now $L(X)\mathrel{\dot{\cup}} O_n^*$ is the minimal digraph $W\in \mathcal{D}$ with the following conditions: $L(X)\leq W$, $O_n^*\leq W$, there is no $O_n^*\prec Z$ for which $L_1\leq Z$ and $Z\leq W$. 
(Here we used the fact that $O_n^*$ has so big circles that cannot fit into $X$.) 
Now Lemma \ref{weiurzwe} tells us that the set of the following first-order conditions suffice:
\begin{itemize}
\item $Y \equiv L(X)\mathrel{\dot{\cup}} O_n^*$,
\item $X \leq Y$,
\item $(X,Y) \in\mathfrak{L}$, and
\item  (taken from the end of the statement of Lemma \ref{weiurzwe}:) there exists NO digraph $\overline{X}$ for which:
\begin{itemize}
\item $ X\leq \overline{X}$, $X\equiv \overline{X}$, and 
\item there exists no digraph $Z$ for which $ Y \leq Z$, $Z\equiv Y$, $Z\leq L(X)\mathrel{\dot{\cup}} O_n^*$, and $(\overline{X}, Z)\in\mathfrak{L}$.
\end{itemize}
\end{itemize}
 \end{proof}

\begin{definition}\label{iuyxcuieiqiufjsf} Let $G\in \mathcal{D}$ be a digraph having $n$ vertices. Let us denote the vertices of $O_n^*$ with
$$V(O_n^*):=\{v_{i,j}: 1\leq i \leq n, \; 1\leq j\leq n+i\} $$
such that $V(O_{n+i})=\{v_{i,j}: 1\leq j\leq n+i\}$. Let $\underline{v}:=(v^1,\dots,v^n)$ be a tuple of the vertices of $G$. 
Let us define the digraph $G\overset{\underline{v}}{\leftarrow} O_n^*$ the following way:
$$ V(G\overset{\underline{v}}{\leftarrow} O_n^*):=V(G\mathrel{\dot{\cup}} O_n^*), \;\;\; E(G\overset{\underline{v}}{\leftarrow} O_n^*):=E(G\mathrel{\dot{\cup}} O_n^*) \cup \{(v_{i,1},v^i): 1\leq i \leq n\}.$$
\end{definition}

\begin{lemm} The following relation is definable:
\begin{equation}\label{baassppwe}
\{(G,G\overset{\underline{v}}{\leftarrow} O_n^*): G\in\mathcal{D},\; |V(G)|=n \text{ and $\underline{v}$ is a tuple of the vertices of $G$}\}.
\end{equation}
\end{lemm}

\begin{proof} First, we define the relation 
\begin{equation}\label{ffbbzzzt}
\{(G,L(G)\overset{\underline{v}}{\leftarrow} O_n^*): G\in\mathcal{D},\; |V(G)|=n \text{ and $\underline{v}$ is a tuple of the vertices of $L(G)$}\}.
\end{equation}
This relation consists of those pairs $(X,Y)\in \mathcal{D}^2$ for which the following holds. Let $(X,E_n)\in\mathfrak{E}$. From $X$, $L(X)$ is definable. Hence, with the relation (\ref{yyaajjjqwe}), $L(X)\mathrel{\dot{\cup}} O_n^*$ is definable. 
Now $Y$ is minimal with the following properties:
\begin{itemize}
\item $L(X)\mathrel{\dot{\cup}} O_n^* \leq Y$ and $(Y, L(X)\mathrel{\dot{\cup}} O_n^*)\in \mathfrak{E}$.
\item There is no $L(X)\prec Z$ for which $(L(X),Z)\in\mathfrak{E}$ and $Z\leq Y$. 
\item There is no $O_n^*\prec Z$ for which $(O_n^*,Z)\in\mathfrak{E}$ and $Z\leq Y$.
\item For all $O_i\in \mathcal{O}_{\cup}$, $O_i\leq O_n^*$ implies $\male_i^L\leq Y$.
\item There are no $O_i, O_j\in \mathcal{O}_{\cup}$ for which $O_i\neq O_j$, $O_i, O_j\leq O_n^*$ and $\male_{i,j}^L\leq Y$.
\end{itemize}
   
Finally, the relation (\ref{baassppwe}) consists of those pairs $(X,Y)\in \mathcal{D}^2$ which satisfy the following conditions. 
Let $(X,E_n)\in\mathfrak{E}$ again. 
Then $Y$ satisfies: there exists $L(X)\overset{\underline{v}}{\leftarrow} O_n^*$ for which
\[
(L(X)\overset{\underline{v}}{\leftarrow} O_n^*, Y)\in \mathfrak{M}, \;\;\; X \mathrel{\dot{\cup}} O_n^*\leq Y\leq L(X)\overset{\underline{v}}{\leftarrow} O_n^*, \;\;\; (X,Y)\in\mathfrak{L}.
\] 
 \end{proof}

\begin{definition}\label{748298471239874} Let $v_1$ and $v^1$ denote the vertices of $\male_i$ and $\male_j$ with degree 1. Let us define $\male_i\to\male_j$ the following way:
$$V(\male_i\to\male_j):=V(\male_i \mathrel{\dot{\cup}} \male_j), \;\; E(\male_i\to\male_j):=E(\male_i \mathrel{\dot{\cup}} \male_j)\cup \{(v_1,v^1)\}.$$
\end{definition}

\begin{lemm} The following relation is definable:
$$ \{(\male_i\to\male_j,E_i,E_j):1<i,j, \;\; i\neq j\}. $$
\end{lemm}

\begin{proof} The relation above consists of those pairs $(X,Y,Z)\in \mathcal{D}^3$ which satisfy the following. $Y,Z\in \mathcal{E}$, $E_2\leq Y, Z$ and $Y\neq Z$, meaning $Y=E_i$, $Z=E_j$, where $1<i, j$ and $i\neq j$. 
Now $O_i\mathrel{\dot{\cup}} O_j\mathrel{\dot{\cup}} E_1$ can be similarly defined as in Lemma \ref{apoowjsbc}.
From this, $O_i\mathrel{\dot{\cup}} O_j\mathrel{\dot{\cup}} E_2$ is the only digraph $W\in \mathcal{D}$ for which $O_i\mathrel{\dot{\cup}} O_j\mathrel{\dot{\cup}} E_1\prec W$ and 
$(W,O_i\mathrel{\dot{\cup}} O_j\mathrel{\dot{\cup}} E_1)\notin\mathfrak{E}.$
Now $O_i\mathrel{\dot{\cup}} O_j\mathrel{\dot{\cup}} L_2$ is the only digraph $W\in \mathcal{D}$ for which there exists $V\in\mathcal{D}$ such that
$$O_i\mathrel{\dot{\cup}} O_j\mathrel{\dot{\cup}} E_2\prec V\prec W,\;\;\; L_2\leq W, \;\; O_{i,L}\nleq W \;\text{ and }\; O_{j,L}\nleq W.  $$
$\male_i^L \mathrel{\dot{\cup}} \male_j^L$ is the only digraph $W\in \mathcal{D}$ for which there exists $V\in\mathcal{D}$ such that
$$O_i\mathrel{\dot{\cup}} O_j\mathrel{\dot{\cup}} L_2\prec V\prec W, \;\;\; \male_i^L\leq W, \;\; \male_j^L\leq W, \; \text{ but }\; \male_{i,j}^L\nleq W. $$
Let $I$ denote the digraph
\begin{equation}\label{hxgetrirow}
V(I)=\{u,v\}, \;\; E(I):=\{(u,v),(u,u),(v,v)\}.
\end{equation}
The set
\begin{equation}\label{yylalasdf}
\{\male_i^L\to\male_j^L,\male_j^L\to\male_i^L\}
\end{equation}
consists of those $W\in \mathcal{D}$ for which 
$\male_i^L \mathrel{\dot{\cup}} \male_j^L \prec W$, $I\leq W$.
The digraph $\male_i^L \mathrel{\dot{\cup}} E_1$ is defined as usual.
From this, the digraph $\male_i^L \mathrel{\dot{\cup}} L_1$ is definable as the only $W\in \mathcal{D}$ for which $\male_i^L \mathrel{\dot{\cup}} E_1 \prec W$, $L_2\leq W$ and there is no $V\in\mathcal{D}$ such that $\male_i^L\prec V$, $L_2\leq V$ and $V\leq W$.
Let $v$ denote the vertex of $\male_i^L$ that has a loop on it and let $x$ be the only vertex of $L_1$. 
Let $\male_i^{L\to}$ and $I^*$ be the following digraphs: 
$$V(\male_i^{L\to}):=V(\male_i^L \mathrel{\dot{\cup}} L_1), \;\; E(\male_i^{L\to}):=E(\male_i^L \mathrel{\dot{\cup}} L_1)\cup \{(v,x)\}$$
\begin{equation}\label{ncbhzuippqw}
V(I^*):=\{u,v,w\}, \;\; E(I^*):=\{(v,v),(w,w),(u,v),(v,w)\}. 
\end{equation}
Now $\male_i^{L\to}$ is the only digraph $W\in\mathcal{D}$ for which $\male_i^L \mathrel{\dot{\cup}} L_1 \prec W$ and $I^*\leq W$. 
From the set (\ref{yylalasdf}) we can choose $\male_i^L\to\male_j^L$ with the fact
$$\male_i^{L\to} \leq \male_i^L\to\male_j^L, \;\;\; \male_i^{L\to} \nleq \male_j^L\to\male_i^L.$$ Finally, $X=M(\male_i^L\to\male_j^L)$.
 \end{proof}

\begin{lemm}\label{74678719187892643} The following relation and set are definable:
\begin{equation}\label{mazstgf}
\{(O_{i,i},E_i):1<i\}, \;\; \{O_{i,i}:1<i\},
\end{equation}
where $O_{i,i}:=O_i\mathrel{\dot{\cup}} O_i$.
\end{lemm}

\begin{proof} The relation (\ref{mazstgf}) consists of those pairs $(X,Y)\in \mathcal{D}^2$ for which the following holds. $Y\in\mathcal{E}$, meaning $Y=E_i$. $X\in \mathcal{O}$, $(X,E_{2i})\in\mathfrak{E}$, and from the set $\mathcal{O}_{\cup}$, $O_i$ is the only element that is embeddable into $X$. The corresponding set can now be easily defined.
 \end{proof}

\begin{lemm} The following relation is definable:
$$\{(O_i^*\mathrel{\dot{\cup}} O_{j,L}^*,E_i,E_j):1<i, j\}.$$
\end{lemm}

\begin{proof} The relation above is the set of triples $(X,Y,Z)\in \mathcal{D}^3$ which satisfy the following. $Y,Z\in \mathcal{E}$, $E_2\leq Y, Z$, meaning $Y=E_i$, $Z=E_j$, where $1<i, j$. Now $O_i^*\mathrel{\dot{\cup}} O_{j}^*$ is the digraph $W$ satisfying the following:
\begin{itemize}
\item $W\in \mathcal{O}$
\item If $E_x,E_y\in \mathcal{E}$ satisfy $(O_i^*,E_x)\in \mathfrak{E}$ and $(O_j^*,E_y)\in \mathfrak{E}$, then $(W,E_{x+y})\in\mathfrak{E}$.
\item For all $O_n\in\mathcal{O}_{\cup}$ that satisfy $O_n\leq O_i^*$ or $O_n\leq O_j^*$, $O_n\leq W$  holds.
\item For all $O_n\in\mathcal{O}_{\cup}$ which satisfy $O_n\leq O_i^*$ and $O_n\leq O_j^*$, $O_{n,n}\leq W$ holds.
\end{itemize}
Finally, $X$ is the minimal digraph with $O_i^*\mathrel{\dot{\cup}} O_{j}^*\leq X\leq L(O_i^*\mathrel{\dot{\cup}} O_{j}^*)$ and $O_{j,L}^*\leq X$.
 \end{proof}

\begin{definition}\label{75387671982738788} Let us denote the vertices of $O_n^*$ by 
$$V(O_n^*):=\{v_{i,j}: 1\leq i \leq n, \; 1\leq j\leq n+i\} $$
such that the circle $O_{n+i}$ consists of $\{v_{i,j}: 1\leq j\leq n+i\}$. Similarly, let us denote the vertices of $O_{m,L}^*$ by
$$V(O_{m,L}^*):=\{v^{i,j}: 1\leq i \leq m, \; 1\leq j\leq m+i\} $$
such that the circle $O_{m+i,L}$ consists of $\{v^{i,j}: 1\leq j\leq m+i\}$ and the loops are on the vertices $\{v^{i,1}: 1\leq i\leq m\}$.  
For a map $\alpha: [n]\to [m]$, we define the digraph $F_{\alpha}(n,m)$ as
$$V(F_{\alpha}(n,m)):=V(O_n^*\mathrel{\dot{\cup}} O_{m,L}^*),$$
$$E(F_{\alpha}(n,m)):=E(O_n^*\mathrel{\dot{\cup}} O_{m,L}^*)\cup \{(v_{i,1},v^{\alpha(i),1}):1\leq i\leq n\}.$$
Let
$$\mathcal{F}(n,m):=\{F_{\alpha}(n,m): \; \alpha: [n]\to [m]\}.$$
\end{definition}

\begin{lemm}The following relation is definable:
\begin{equation}\label{1kajszt}
\{(F_{\alpha}(n,m),E_n,E_m):1\leq n,m, \;\; \alpha: [n]\to [m]\}.
\end{equation}
\end{lemm}

\begin{proof} The relation above consists of those triples $(X,Y,Z)\in \mathcal{D}^3$ that satisfy the following. $Y,Z\in \mathcal{E}$, meaning $Y=E_n$, $Z=E_m$, where $1\leq n, m$. Now $X$ is a minimal digraph with the following conditions:
\begin{itemize}
\item $O_n^*\mathrel{\dot{\cup}} O_{m,L}^*\leq X$ and
$(O_n^*\mathrel{\dot{\cup}} O_{m,L}^*,X)\in\mathfrak{E}\cap\mathfrak{L}$.
\item $O_i\leq O_n^*$ implies $\male_i^L\leq X$.
\item $O_i\leq O_n^*\mathrel{\dot{\cup}} O_{m,L}^*$ implies there is no $W\in\mathcal{O}_i^{\to}$, for which $W\leq X$.
\item There is no $V$ for which $V\leq X$ and $\male_i\prec V$, such that $\male_i\leq X$ and $\male_i^L \nleq V$.
\item There is no $\male_i^L\prec V$ for which $V\leq X$ and $L_2 \leq V$. 
\end{itemize}
 \end{proof}

\begin{lemm}\label{7682} The following relation is definable:
\begin{equation}\label{jajybccvtep}
\{(F_{\id_{[n]}}(n,n),E_n,E_n):1\leq n\}.
\end{equation}
\end{lemm}

\begin{proof} The relation in question consists of those triples $(X,Y,Z)\in (\ref{1kajszt})$ for which $Y=Z\in\mathcal{E}$ and for $i,j \geq 2$ we have
\[O_{i\to j}\leq X \Rightarrow E_i=E_j. \]
 \end{proof}

\begin{lemm}The following relation is definable:
\begin{equation}\label{vyniurteaal}
\{(F_{\alpha}(n,m),F_{\beta}(m,l),F_{\beta\circ\alpha}(n,l), E_n,E_m,E_l):1\leq n,m,l, \;\; \alpha: [n]\to [m], \;\; \beta: [m]\to [l]\}.
\end{equation}
\end{lemm}

\begin{proof} The relation in question is the set of those 6-tuples $(X_1,\dots,X_6)\in\mathcal{D}^6$ which satisfy the following. 
$X_4,X_5,X_6\in \mathcal{E}$, meaning $X_4=E_n$, $X_5=E_m$ and $X_6=E_l$ where $1\leq n, m,l$. 
$X_1\in \mathcal{F}(n,m)$, $X_2\in \mathcal{F}(m,l)$ and $X_3\in \mathcal{F}(n,l)$. Finally:
\[(O_{i \to j}\leq X_1\text{ and } O_{j\to k}\leq X_2) \; \Rightarrow \;  O_{i \to k}\leq X_3.\]
 \end{proof}

\begin{definition}\label{723789811980012} There is a bijection between the digraphs $G\overset{\underline{v}}{\leftarrow} O_n^*$ and the elements of $\text{ob}(\mathcal{CD})$. Let us observe that the vertices of $G$ are labeled with the circles $O_{n+1},O_{n+2}, \dots, O_{2n}$ in $G\overset{\underline{v}}{\leftarrow} O_n^*$. On the other hand, in $\text{ob}(\mathcal{CD})$, they are labeled with $1,\dots,n$. The element of $\text{ob}(\mathcal{CD})$ that corresponds to $G\overset{\underline{v}}{\leftarrow} O_n^*$ will be denoted by $(G\overset{\underline{v}}{\leftarrow} O_n^*)_{\mathcal{CD}}$ from now on.
\end{definition}

\begin{lemm}\label{5634782}The following relation is definable:
\begin{equation}\label{ghajsvxclhhg}
\begin{split}
\{(X, F_{\alpha}(n,m), Y)\in\mathcal{D}^3: \; &X=G\overset{\underline{v}}{\leftarrow} O_n^*, \; Y=H\overset{\underline{w}} {\leftarrow} O_m^* \text{ for some $\underline{v}$ and $\underline{w}$}, \text{ and}\\
&((X)_{\mathcal{CD}}, \alpha, (Y)_{\mathcal{CD}})\in\mathrm{hom}((X)_{\mathcal{CD}},(Y)_{\mathcal{CD}})\} 
\end{split}
\end{equation}
\end{lemm}

\begin{proof} The relation in question is the set of those pairs $(X,F,Y)\in \mathcal{D}^3$ which satisfy the following. 
There exist $G$ and $H$ such that $(G,X)\in (\ref{baassppwe})$ and $(H,Y)\in (\ref{baassppwe})$. Finally, $F$ satisfies
\begin{itemize}
\item $(F,E_n,E_m)\in (\ref{1kajszt})$,
\item $(\male_i\to\male_j\leq G\overset{\underline{v}}{\leftarrow} O_n^*(=X),\;\; O_i,O_j\leq O_n^*\text{ and } O_{i\to k},O_{j\to l}\leq F) \Longrightarrow$ \\
$((O_k\neq O_l \text{ and } \male_k\to\male_l\leq H\overset{\underline{w}}{\leftarrow} O_m^*) \;\; \vee \;\; (O_k=O_l \text{ and } \male_k^L\leq H\overset{\underline{w}}{\leftarrow} O_m^*))$, 
\item $(\male_i^L\leq G\overset{\underline{v}}{\leftarrow} O_n^*, \;\; O_i\leq O_n^*\text{ and } O_{i\to k}\leq F) \Rightarrow \male_k^L\leq H\overset{\underline{w}}{\leftarrow} O_m^*.$
\end{itemize}
 \end{proof}

The proof of Theorem \ref{9ahgfa} is now properly prepared for, we only need to put the pieces together.

\begin{proof}[Proof of the main theorem: Theorem \ref{9ahgfa}]
We have already seen in Section \ref{pakdhtgft} that all relations first-order definable in $\mathcal{D}'$ are defiable in $\mathcal{CD}'$ as well. So we only need to deal with the converse. 
We wish to build a copy of $\mathcal{CD}'$ inside $\mathcal{D}'$ so that all things we can formulate in the first-order language of $\mathcal{CD}'$ becomes accesible in its model in $\mathcal{D}'$.
Let the set of objects be 
$$\{G\overset{\underline{v}}{\leftarrow} O_n^*: G\in\mathcal{D},\; |V(G)|=n \text{ and $\underline{v}$ is a vector of the vertices of $G$}\},$$
and the set of morphisms be \eqref{ghajsvxclhhg}.
We can define both as Lemma \ref{5634782} shows. 
Identity morphisms can be defined with Lemma \ref{7682}. 
For the triples
$$(X_1, Z_1, Y_1),(X_2, Z_2, Y_2),(X_3, Z_3, Y_3) \in \mathcal{D}^3$$
the condition $(X_i, Z_i, Y_i)\in (\ref{ghajsvxclhhg})$ ensures that there exist $\alpha_i$ such that
$$((X_i)_{\mathcal{CD}}, \alpha_i, (Y_i)_{\mathcal{CD}})\in \mathrm{hom}((X_i)_{\mathcal{CD}},(Y_i)_{\mathcal{CD}}).$$
Moreover, if we suppose  $Y_1=X_2$, $X_3=X_1$, $Y_3=Y_2$ and that there exists a 6-tuple in (\ref{vyniurteaal}) of the form $(Z_1, Z_2, Z_3,\ast,\ast,\ast)$, we have forced 
$$((X_1)_{\mathcal{CD}}, \alpha_1, (Y_1)_{\mathcal{CD}})
((X_2)_{\mathcal{CD}}, \alpha_2, (Y_2)_{\mathcal{CD}})
 =((X_3)_{\mathcal{CD}}, \alpha_3, (Y_3)_{\mathcal{CD}}) .$$ 
The four constants in $\mathcal{CD}'$ require 4 digraphs, say, 
\begin{equation}\label{cnjusiuduh}
C_1, C_2, C_3, \text{ and } C_4
\end{equation}
of $\mathcal{D}'$ to be defined such that
$$ (C_1)_{\mathcal{CD}}={\bf E}_1, \; (C_2)_{\mathcal{CD}}={\bf I}_2$$
and $C_3$, and $C_4$ are the elements of the set $\mathcal{F}(1,2)$. Now we have all the ``tools'' accesible in $\mathcal{CD}'$. Finally, the relation (\ref{baassppwe}) lets us ``convert'' the elements of $\mathcal{D}'$ and $\mathcal{CD}'$ back and forth. We are done.
 \end{proof}

\section{Table of notations}\label{jelolesj}

  \begin{longtable}{| l | p{2cm} | l | }

  \hline
 Notation & Definition, theorem, etc. & Page number \\ \hline
$\dot{\cup}$ & \ref{349587897819}  & \pageref{349587897819}  \\ \hline
$\sqsubset$, $ \sqsubseteq$, $\equiv$, $\equiv_G^C$ & \ref{7538761239874}  & \pageref{7538761239874}  \\ \hline

$(.)_{\mathcal{CD}}$ & \ref{723789811980012} & \pageref{723789811980012} \\ \hline
$\male_n$ & \ref{deffiugraf} & \pageref{deffiugraf} \\ \hline
$\male_n^L$ & \ref{deffiugraf} & \pageref{deffiugraf}  \\ \hline
$\male_{i,j}^L$ & \ref{13541235} & \pageref{13541235}  \\ \hline
$\male_i\to\male_j$ & \ref{748298471239874} & \pageref{748298471239874}  \\ \hline
$A$ & \ref{xcmvnyxjdaskdqiwo} & \pageref{xcmvnyxjdaskdqiwo} \\ \hline
${\bf A}$ &  & \pageref{xcvnsbndasiiwq}  \\ \hline
$(A,\alpha,B)$ &  & \pageref{yxcmmnbcvn16235} \\ \hline 
$\mathcal{CD}$ &  & \pageref{qoiue2833} \\ \hline 
$\mathcal{CD}'$ & \ref{xmcvnbsjdiqawe} & \pageref{xmcvnbsjdiqawe} \\ \hline 
$\mathcal{D}$ &  & \pageref{xcvshgw8} \\ \hline
$\mathcal{D}'$ & \ref{xcmvnyxjdaskdqiwo} & \pageref{xcmvnyxjdaskdqiwo} \\ \hline 
$\mathfrak{E}$ & \ref{xmvbyxhu8231} & \pageref{xmvbyxhu8231} \\ \hline 
$\mathfrak{E}_+$ & \ref{238971875642} & \pageref{238971875642} \\ \hline 
$E(G)$ & & \pageref{qweoweru3} \\ \hline
${\bf E}_1$ &\ref{xmcvnbsjdiqawe} & \pageref{xmcvnbsjdiqawe} \\ \hline
$E_n$ & \ref{defE_n} & \pageref{defE_n} \\ \hline
$F_{\alpha}(n,m)$ & \ref{75387671982738788} & \pageref{75387671982738788} \\ \hline
$\mathcal{F}(n,m)$ & \ref{75387671982738788} & \pageref{75387671982738788} \\ \hline
$E_n$ & \ref{defE_n} & \pageref{defE_n} \\ \hline
${\bf f}_1$, ${\bf f}_2$ & \ref{xmcvnbsjdiqawe} & \pageref{xmcvnbsjdiqawe} \\ \hline
$F_n$ & \ref{defE_n} & \pageref{defE_n} \\ \hline
$G^T$ &  & \pageref{aqweiqportw8} \\ \hline
$G\overset{\underline{v}}{\leftarrow} O_n^*$ & \ref{iuyxcuieiqiufjsf} & \pageref{iuyxcuieiqiufjsf} \\ \hline
$\text{hom}(A, B)$ &  & \pageref{1382721} \\ \hline
${\bf I}_2$ & \ref{xmcvnbsjdiqawe} & \pageref{xmcvnbsjdiqawe} \\ \hline
$I_n$ & \ref{defIOL} & \pageref{defIOL} \\ \hline
$\mathfrak{L}$ & \ref{xmvbyxhu8231} & \pageref{xmvbyxhu8231} \\ \hline
$L_n$ & \ref{defIOL} & \pageref{defIOL}  \\ \hline
$L(.)$ & \ref{defLfuggveny} & \pageref{defLfuggveny}  \\ \hline
$\mathcal{L}(.)$ & \ref{defLfuggveny} & \pageref{defLfuggveny}  \\ \hline
$\mathfrak{M}$ & \ref{ppghhfjasd} & \pageref{ppghhfjasd}  \\ \hline
$M(.)$ & \ref{defMfuggveny} & \pageref{defMfuggveny}  \\ \hline
$\mathcal{M}(.)$ & \ref{defMfuggveny} & \pageref{defMfuggveny}  \\ \hline
$\mathcal{O}$ & \ref{123897834611} & \pageref{123897834611}  \\ \hline
$\mathcal{O}_{\cup}$ & \ref{ppghhfjasd} & \pageref{ppghhfjasd}  \\ \hline
$O_n$ & \ref{defIOL} & \pageref{defIOL}  \\ \hline
$\mathcal{O}_n^{\to}$ & \ref{defOnto} & \pageref{defOnto}  \\ \hline
$O_n^*$ & \ref{983489887772543214} & \pageref{983489887772543214}  \\ \hline
$O_{n,L}^{*}$ & \ref{apoowjsbc} & \pageref{apoowjsbc}  \\ \hline
$O_{n,L}$ & \ref{defonl} & \pageref{defonl}  \\ \hline
$O_{i,i}$ & \ref{74678719187892643} & \pageref{74678719187892643}  \\ \hline
$O_{i \to j}$ & \ref{756728767198470} & \pageref{756728767198470}  \\ \hline
$\text{ob}(\mathcal{CD})$ &  & \pageref{dfjhawkd}  \\ \hline
$V(G)$ &  & \pageref{qwqepwrpdg}  \\ \hline
$\mathcal{W}$ & \ref{7358712837283791}  & \pageref{7358712837283791}  \\ \hline
  \end{longtable}

{\it Acknowledgements.} The results of this paper were born in an MSc thesis. The author thanks Mikl\'{o}s Mar\'{o}ti, who as his supervisor gave him this research topic.

\bibliographystyle{spmpsci}      
\bibliography{definability_II}

\end{document}